\documentclass[12pt,a4paper]{amsart}

\pdfoutput=1

\input{mathdefs.sty}
\input{jens_defs.sty}

\renewcommand{\myboldfont}{\mathbb}

\usepackage[margin=2.4cm,includehead]{geometry}

\usepackage{bm}
\usepackage{enumerate}

\newtheorem{theorem}{Theorem}
\newtheorem{lemma}[theorem]{Lemma}

\newtheorem{prop}[theorem]{Proposition}
\newtheorem{cor}[theorem]{Corollary}

\newcounter{theremark}
\numberwithin{equation}{subsection}

\usepackage{graphicx}

\usepackage[font=small,labelfont=bf]{caption}

\begin{document}

\title[Sequences modulo one]{Sequences modulo one: convergence of local statistics}
\author{Ilya Vinogradov}
\date{\today}


\thanks{The research leading to these results has received funding from the European Research Council under the European Union's Seventh Framework Programme (FP/2007-2013) / ERC Grant Agreement n. 291147. }

\begin{abstract}We survey recent results beyond equidistribution of sequences modulo one. We focus on the sequence of angles in a Euclidean lattice in $\R^2$ and on the sequence $\sqrt n\bmod1 $. 
\end{abstract}

\maketitle

\tableofcontents

\section{Introduction\label{sec:intro}}

The study of randomness in number theory has been very fruitful in recent years, with new results in areas ranging from the M\"obius function to values of forms at integer points. 
To prove that a deterministic sequence is random in a certain sense is typically harder than showing that another sequence lacks randomness; indeed there are very few examples of number theoretic sequences that are truly indistinguishable from a sequence of random variables. Many sequences whose statistical properties are well understood are connected to dynamical systems. In this case the problem can often be reduced to sampling an observable along a trajectory of this dynamical system, and statistical properties (or lack thereof) are inherited from the underlying dynamical system. 


Sarnak \cite{sarnak_mobius} conjectures that the M\"obius sequence is disjoint from zero entropy systems following the heuristic of the M\"obius randomness principle \cite[Sect.~13]{iwaniec_kowalski_analytic_2004}. The conjecture is known to hold for a large class of systems, most notably for the horocycle flow on $\sltz\quot \sltr$ \cite{bourgain_distjointness_2011}, and suggests that the M\"obius function possesses a form of randomness in agreement with the Riemann Hypothesis. This phenomenon is particularly interesting since many sequences built out of the M\"obius function  possess very little randomness \cite{peckner_uniqueness_2012, cellarosi_sinai_ergodic_2013, cellarosi_vinogradov_2013, abdalaoui_dynamical_2013}. Indeed, $\{\mu^2(n)\}_{n\in \N}$ is generic for a translation on a compact abelian group. 


Given a real quadratic form, the set of its values at the integers can also be studied as a random object. If the form is generic, i.e.~badly approximable by rational forms, numerical experiments suggest that the fine-scale statistics are the same as those of a Poisson point process. The only result to-date in this direction is the proof of the convergence of the pair correlation function \cite{sarnak_values_1997, Eskin05quadraticforms, marklof_pair_correlation_2003, marklof_pair_correlation_II_2002, margulis_quantitative_2011}. The convergence of higher-order correlation functions has only been established in the case of generic (in measure) positive definite quadratic forms in many variables \cite{vanderkam_values_1999, vanderkam_pair_correlation_1999, VanderKam00}. The situation is similar in the problem of fine-scale 
statistics for the fractional parts of the sequence $\{n^2\alpha \bmod 1\}_{n\in\N}$, where we expect the local 
statistics to converge to those of a Poisson point process (after appropriate rescaling), provided $\alpha$ is badly approximable by rationals. As in the case of binary quadratic forms, we have results for the two-point correlation function 
\cite{rudnick_pair_1998, marklof_equidistribution_2003, heath-brown_pair_2010}. Convergence of the gap distribution for well approximable $\alpha$ along a subsequence is established in \cite{rudnick_sarnak_zaharescu_2001}. 


\begin{figure}
\begin{centering}\includegraphics[width=0.9\textwidth]{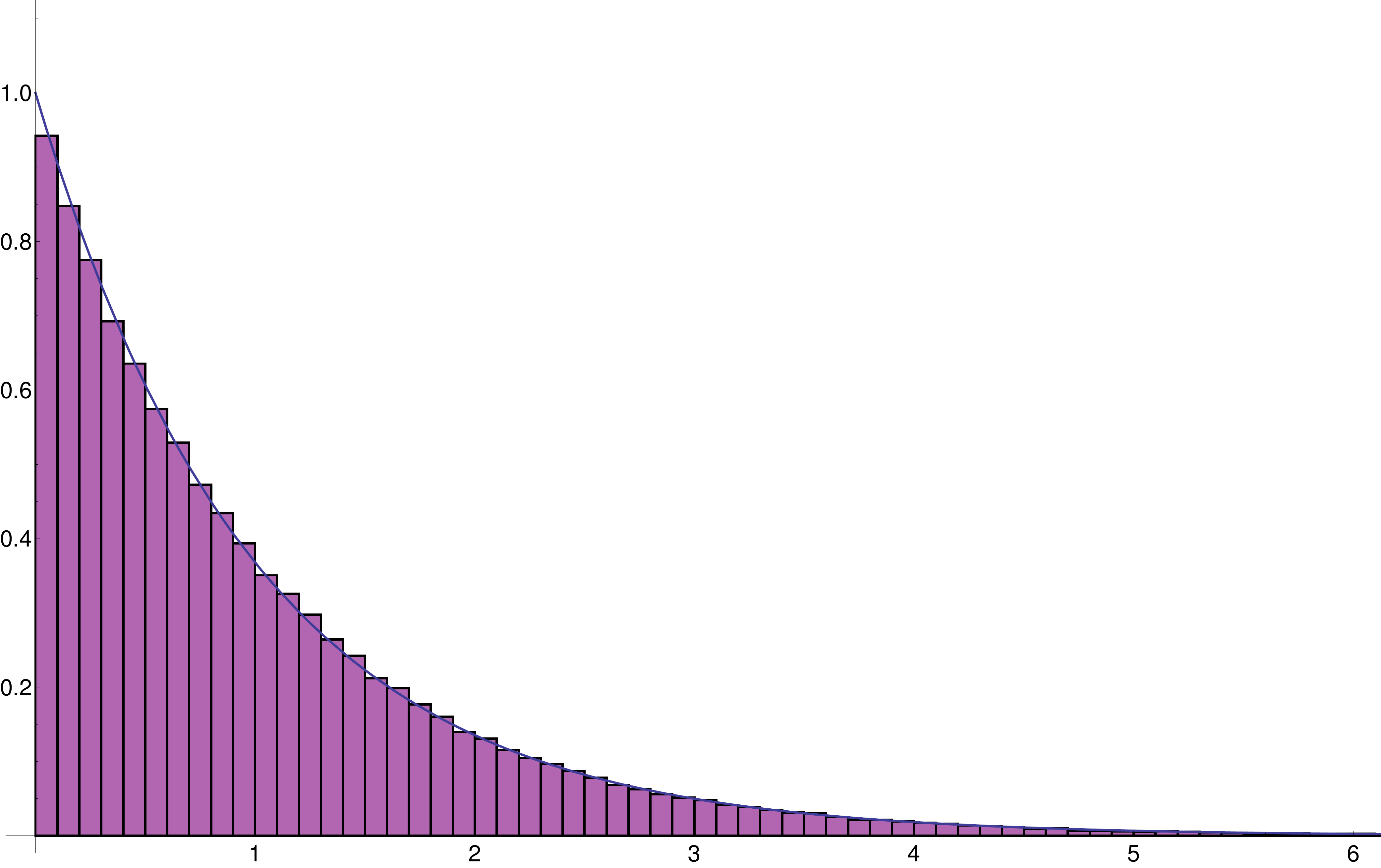}\par\end{centering}
\caption{Gap distribution of the fractional parts of $n^{1/3}$ with $n\leq 2\times 10^5$.  }\label{fig:n13gap}
\end{figure}
\begin{figure}
\begin{centering}\includegraphics[width=0.9\textwidth]{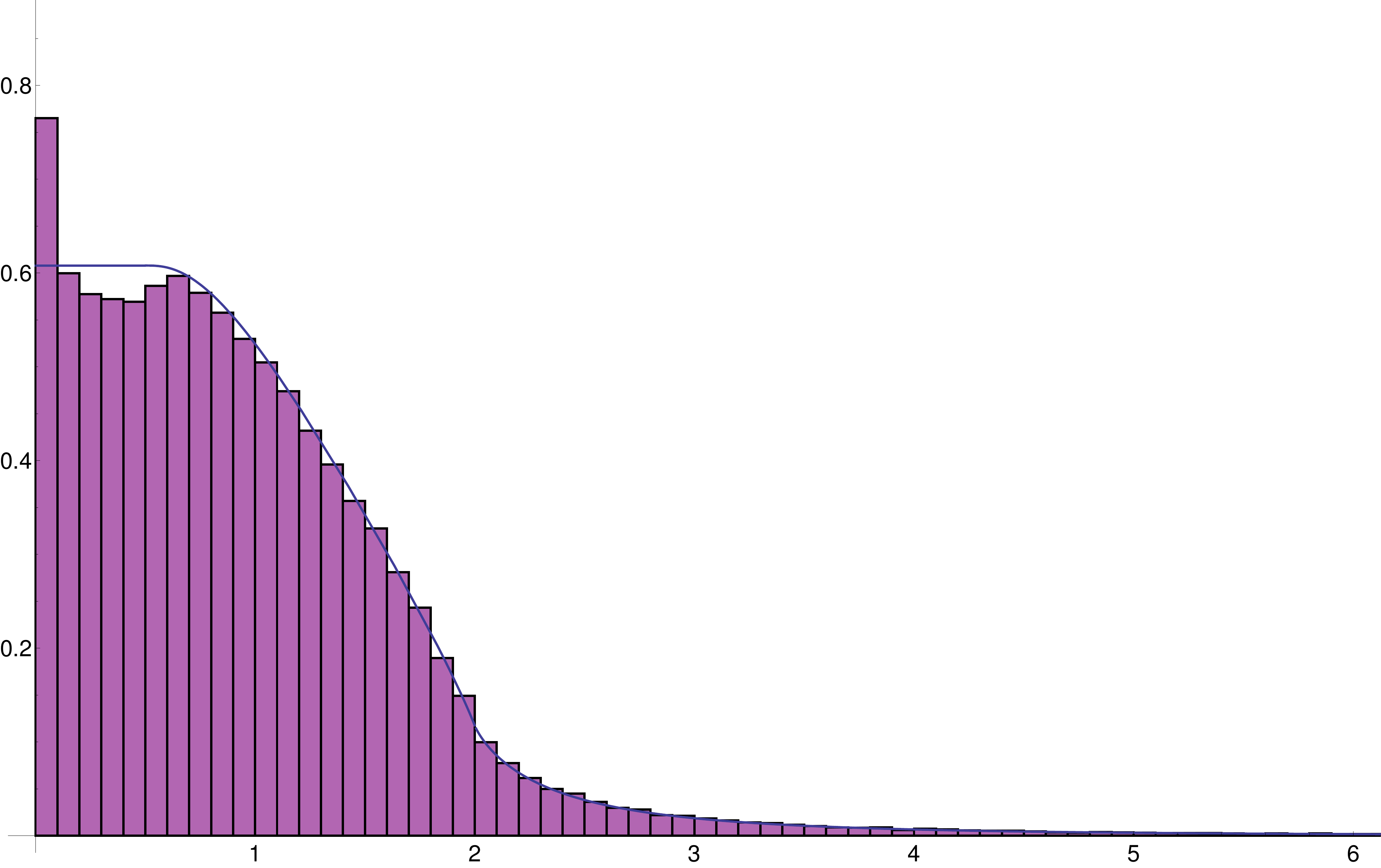}\par\end{centering}
\caption{Gap distribution of the fractional parts of $\sqrt n$ with $n\leq 2\times 10^5$.  }\label{fig:n12gap}
\end{figure}

Another object of study is the sequence $\{n^\alpha\bmod 1\}_{n\in\N}$ for fixed $\alpha\in(0,1)$, which is easily seen to be uniformly distributed.  Numerical experiments (cf.~ Figures \ref{fig:n13gap} and \ref{fig:n12gap}) suggest that the gaps in this sequence converge to the exponential distribution with parameter $1$ as $n\to\infty$, which is the distribution of waiting times in a Poisson process. The only known exception is the case $\alpha=1/2$. Here Elkies and McMullen \cite{ElkiesMcM04} proved that the limit gap distribution exists and is given by a piecewise analytic function with a power-law tail (Sinai \cite{Sinai13} proposed a different solution to show convergence).

In this note we survey recent results concerning two sequences, $\sqrt n$ modulo $1$ and the set of directions for points of an affine lattice. For each sequence we show that the two-point correlation function  and is Poisson, following \cite{el-baz_two-point_2013, distribution_of_directions}. This convergence also applies to more general mixed moments and depends on the Diophantine properties of the translation parameter in the case of affine lattices. The appearance of Diophantine conditions for the convergence of moments is reminiscent of the same phenomenon in the quantitative Oppenheim conjecture, in particular the pair correlation problem for the values of quadratic forms at integers \cite{Eskin05quadraticforms, marklof_pair_correlation_II_2002, marklof_pair_correlation_2003}. The techniques we use here generalize the approach in \cite{marklof_pair_correlation_II_2002, marklof_mean_square_2005}.

Even more recently rates for of convergence for local statistics of these sequences have been established \cite{strombergsson_effective_2013, browning_effective_2013} but we do not discuss this work here. 

The plan of these lecture notes is as follows. In Section \ref{sec:statistics} we define local statistics of sequences modulo one and give their limits in the random case. The sequence of directions in a Euclidean lattice is analyzed in Section \ref{sec:directions}, and Section \ref{sec:sqrtn} is dedicated to the sequences of square roots modulo one. 

\textbf{Acknowledgements.} The author is grateful to Jens Marklof and Daniel El-Baz for comments on the text.

\section{Statistics for sequences modulo 1\label{sec:statistics}}


\subsection{Uniform distribution}

Let $\Xi=\{\xi_n\}_{n\ge 1}$ be a fixed sequence in $[0,1)=\R/\Z$. We present several ways of comparing its long term behavior to that of a typical realization of a sequence of independent uniformly distributed (IUD) random variables on $[0,1)$. The first and crudest such measure is \emph{uniform distribution}: a sequence is said to be \emph{uniformly distributed modulo $1$} if for any interval $[a,b)\subset [0,1)$, 
\beq
\lim_{N\to\infty} \frac{\# \{n\in[1,N]\cap \Z : \xi_n \in[a,b)\}}{N} = b-a. \label{eq:uniform_distribution}
\eeq
This condition implies that each interval gets its fair share of points of the sequence. It follows from the Law of Large Numbers that almost any realization of an IUD sequence is uniformly distributed. The same is true of many interesting fixed sequences, such as 
\begin{enumerate}
 \item $\xi_n=n\alpha \bmod 1$ for $\alpha\in\R\setminus \Q$, 
 \item  $\xi_n=n^2\alpha \bmod 1$ for $\alpha\in\R\setminus \Q$,
 \item $\xi_n=n^\alpha\bmod 1$ for $\alpha \in \R^+\setminus \Z$,
 \item $\xi_n = 2^n \alpha \bmod 1$ for almost every $\alpha\in \R$, 
 \item $\xi_n = \alpha^n \bmod 1$ for almost every $\alpha>1$. 
\end{enumerate}
The Weyl equidistribution criterion is helpful in confirming equidistribution sequences 1, 2, 3, and 5, while proofs using ergodic theory are known for 1, 2, and 4. 
The reader is advised to consult references \cite{kuipers_niederreiter_uniform, boshernitzan_uniform_1994, cornfeld_ergodic_1982} for proofs. 

\subsection{Poisson scaling regime}
The fact that uniformly distributed sequences are so diverse suggests that a finer tool for studying such sequences is needed.  The measure of randomness we introduce amounts to studying visits to a shrinking interval that contains finitely many points on average. For an interval $I\subset \R$, define $X_N^\Xi(\cdot,I)\colon [0,1) \to \Z_{\ge 0}$ by 
\begin{align}
 \label{eq:defX}
 X_N^\Xi(x,I) = \sum_{\substack{1\le n\le N\\m\in\Z}} \chi_{I} (N(\xi_n -x + m)),
\end{align}
where $\chi_{I}$ is the indicator function of the set $I$ (we will suppress the dependence on $\Xi$ for brevity). This quantity should be thought of as a random variable realized on $[0,1)$ with $x$ distributed according to the Lebesgue measure. The sum over $m$ typically consists of one nonzero term; it ensures that the signed distance between $\xi_n$ and $x$ is measured on the circle $[0,1)$ with endpoints identified. The definition \eqref{eq:defX} is analogous to \eqref{eq:uniform_distribution} in the sense that it counts the number of visits to the interval $[x,x+I/N)$ up to time $N$.  This combination of interval length being the reciprocal of the number of points is known as the \emph{Poisson scaling regime} (cf.\ \cite{marklof_survey} for a discussion of the Poisson and other scaling regimes). It is also useful to define a smooth version of $X_N(x,I)$. For $f \colon \R \to \R_{\ge 0}$ of compact support and smooth away from a Lebesgue null set, let $X_N(\cdot, f) \colon [0,1)\to \R_{\ge0}$ be defined by 
\begin{align}\label{eq:defXf}
  X_N(x,f) = \sum_{\substack{1\le n\le N\\m\in\Z}} f (N(\xi_n -x+m)).
\end{align}
The natural question is whether the sequence of random variables $X_N(\cdot,f)$ converges in distribution as $N\to \infty.$ That is, does there exist a distribution function $F(\cdot,f)\colon\R_{\ge 0} \to [0,1]$ such that 
\begin{align}
 \label{eq:limiting_distributionf}
 \leb\{x\in [0,1) : X_N(x,f)\le R\} \to F(R,f)
\end{align}
as $N\to \infty$ when $R$ is a point of continuity of $F(\cdot,f)$? The corresponding question for $X_N(x,I)$ is whether there exists $X_\infty(\cdot,I) : \Z_{\ge 0} \to [0,1]$ with $\sum_{r=0}^\infty X_\infty(r,I)=1$ such that 
\begin{align}
 \label{eq:limiting_distributionX}
 \leb\{x\in [0,1) : X_N(x,I)= r \} \to X_\infty(r,I)
\end{align}
for every $r\in\Z_{\ge 0}$ as $N\to\infty$. 

The existence of the limit of $X_N$ very strongly depends on the underlying sequence $\Xi$, and rigorous results are rather scant. We will be concerned with the sequence $\sqrt{n}\bmod 1$ (Section \ref{sec:sqrtn}) and the sequence of directions in an affine Euclidean lattice defined precisely in Section \ref{sec:directions}.

In the case when $\Xi$ is almost any realization of a sequence of IUD's, the answer to both questions above is positive. In fact, it is not difficult to show that $X_\infty(r,I) = e^{-|I|} |I|^r/r!$; that is $X_\infty(\cdot,I)$ is Poisson-distributed with parameter $|I|$, length of $I$. Moreover, one can consider the set $\{N(\xi_n -x \bmod1): 1\le n\le N\}$ as a realization of a point process on $\R$ where $x$ is a Lebesgue-random parameter, the representative for $\xi_n -x \bmod1$ being chosen in the interval $[-1/2,1/2)$. Then, in the case of almost any realization of an IUD sequence, finite-dimensional distributions of the process 
\begin{align}\label{eq:process}
 x\mapsto \{N(\xi_n -x\bmod 1): 1\le n\le N\}
\end{align}
converge to those of a Poisson point process with intensity $1$. 


\subsection{Construction of general statistics}

Popular statistics of the Poisson scaling regime can be defined via $X_N(\cdot,f).$ For $f \colon \R \to \R_{\ge 0}$ of compact support and smooth away from a Lebesgue null set, the \emph{pair correlation function} (or \emph{two-point correlation function}) $R^2_N(f)$ is defined by 
\begin{align}
 \label{eq:pair_def}
 R_N^2(f) = \frac1N \sum_{\substack{1\le n_1\ne n_2 \le N\\m\in\Z}} f(N(\xi_{n_1}-\xi_{n_2}+m)).
\end{align}
Note that the sum typically contains about $N$ nonzero terms, so that $R_N^2(f)$ is $O(1)$ with this normalization. More general $k$-point correlation functions can be defined by considering differences $\xi_{n_j}-\xi_{n_{j+1}}$. Just like the pair correlation function, the $k$-point correlation functions can be expressed in terms of mixed moments of $X_N(\cdot,f)$, but we only consider the case $k=2$. The following Lemma shows how to build the pair correlation function out of $X_N(\cdot,f)$. 

\begin{lemma}\label{lem:X_to_pair}
For compactly supported $f_1, f_2 \colon \R \to\R$ that are almost everywhere continuous set 
\[
 f_1 *' f_2 (w) = \int_\R f_1(w+t) f_2(t)\, dt.
\]
Then, we have 
\begin{align}
 R^2_N(f_1*'f_2) = \int_\T X_N(x,f_1) X_N(x,f_2) \, dx -\int_\T X_N(x,f_1\cdot f_2) \,dx,
\end{align}
for $N$ sufficiently large depending on supports of $f_1$ and $f_2$. 
\end{lemma}

\begin{proof}We have
\begin{multline*}
\int_\T  X_N(x,f_1) X_N(x,f_2) \, dx  = \int_\T \sum_{\substack{1\le n_1,n_2\le N\\m_1,m_2\in\Z}} f_1(N(\xi_{n_1}-x+m_1)) f_2(N(\xi_{n_2}-x+m_2))\, dx\\
 = \sum_{\substack{1\le n_1\ne n_2\le N\\m_1,m_2\in\Z}} \int_\T f_1(N(\xi_{n_1}-x+m_1)) f_2(N(\xi_{n_2}-x+m_2))\, dx 
\\
+ \sum_{\substack{1\le n\le N\\m\in\Z}} \int_\T (f_1\cdot f_2)(N(\xi_{n}-x+m)) \, dx,
\end{multline*}
where we set $m_1=m_2=m$ in the last term assuming $N$ is large. Evaluating both integrals we get
\begin{align*}
& = \sum_{\substack{1\le n_1\ne n_2\le N\\ m\in\Z}} \int_\R f_1(N(\xi_{n_1}-\xi_{n_2}-t+m)) f_2(-Nt) \, dt + \sum_{1\le n\le N} \int_\R (f_1\cdot f_2)(-Nt)\, dt\\
& = \frac1N \sum_{\substack{1\le n_1\ne n_2\le N\\m\in \Z}} \int_\T f_1(N(\xi_{n_1}-\xi_{n_2}+m)+t) f_2(t) \, dt+\int_\R (f_1\cdot f_2)(t)\, dt\\
& = R^2_N(f_1*'f_2) + \int_\T X_N(x,f_1\cdot f_2) \,dx,
\end{align*}
as needed. 
\end{proof}
Pair correlations for general functions (not ``convolutions'' $f_1*'f_2$) can be obtained as limits by an approximation argument (cf.\ Appendix 1 in \cite{distribution_of_directions}). Note also that in the case of IUD's, \[R_N^2(f)\to \int_\R f(s) \, ds\] as $N\to\infty.$ 

Another commonly used measure of randomness that is derivable from $X_N(\cdot,I)$ is the \emph{gap distribution}. Fix $N$, and let $\{\xi_1, \dots, \xi_N\} = \{\xi_1'\le \dots \le \xi_N'\}$ with the same repeats if need be. For $x\ge 0$, let 
\begin{align}\label{eq:gapdist_def}
 \lambda_N(x) = \frac1N \#\{n\in [1,N]\cap \Z : \xi_{n+1}'-\xi_n'< x/N\},
\end{align}
with the obvious interpretation when $n=N$. This is the fraction of gaps that are shorter than $x/N$. Since there are $O(N)$ gaps, the length of an average gap is of order $1/N$, so that this scaling leads to a finite quantity. It is proven in \cite{marklof_survey} that if $X_\infty(0,[0,A])$ exists and satisfies 
\begin{align}\label{eq:gaplimit}
\lim_{A\to0} X_\infty(0,[0,A]) = 1,\quad \lim_{A\to\infty } \frac{dX_\infty(0,[0,A])}{dA} = 0,
\end{align}
then $\lambda_N(A) \to 1+\frac{dX_\infty(0,[0,A])}{dA} = \lambda_\infty(A) $ as $N\to\infty$ at points of continuity of the limit. Thus the limiting behavior of the gap distribution can be understood entirely though the convergence of $X_N$ and the form of $X_\infty$. For comparison, in the case of IUD's, $\lambda_\infty(A) = 1-e^{-A}$, which is the exponential distribution with parameter $1$. This is consistent with the picture of a Poisson point process that arises as the limit of \eqref{eq:process}. 

More general \emph{$k$-neighbor distributions} are constructed using differences $\xi_{n+k}'-\xi_n'$ for $k>1$ in \eqref{eq:gapdist_def}; they count distances to the $k^{\text{th}}$ neighbor, ignoring the first $k-1$ neighbors. The limiting $k$-neighbor distribution, if it exists, is related to the derivative of $X_\infty(k,[0,A])$, analogously to the case $k=0$ in \eqref{eq:gaplimit}. In this sense neighbor distributions can be recovered from $X_N(\cdot,f)$. 

Numerics first performed by Boshernitzan in the 1990s suggest that the limiting gap distribution for  sequences like $\xi_n = n^\alpha \log^\beta n\bmod 1$ (with $\alpha$ and $\beta$ chosen to ensure uniform distribution) exists and is exponential, as in the random setting, save the case $\xi_n = n^{1/2} \bmod 1$ (cf.~Figures \ref{fig:n13gap} and \ref{fig:n12gap}). For the sequence of square roots, the entire limiting point process was understood by Elkies and McMullen \cite{ElkiesMcM04}, and the gap distribution is a non-universal distribution. The problem of local statistics for $n^\alpha \log^\beta n\bmod 1$ (except $\alpha=1/2$, $\beta=0$) is completely open. (See however \cite{marklof_gaps_2013} for the study of the gap distribution of $\log_b n\bmod 1$, which is \emph{not} uniformly distributed.) Another sequence that surprisingly leads to the same point process is the set of directions in an affine lattice with irrational shift, which we discuss in the next section.

\section{Directions in affine lattices\label{sec:directions}}


\subsection{Setup}

In this section we construct a deterministic sequence whose two-point correlation function converges to the Poisson limit, although the limiting process is not Poisson. This sequence is given by the directions of vectors in an affine Euclidean lattice of length less than $T$, as $T\to\infty$. 

Let $\scrL\subset\R^2$ be a Euclidean lattice of covolume one. We may write $\scrL=\Z^2 M_0$ for a suitable $M_0\in\SL(2,\R)$. For $\vecxi=(\xi^1,\xi^2)\in\R^2$, we define the associated affine lattice as $\scrL_\vecxi=(\Z^2+\vecxi)M_0$. Denote by $\scrP_T$ the set of points $\vecy\in\scrL_\vecxi\setminus\{\vecnull\}$ inside the open disc of radius $T$ centered at zero. The number $N(T)$ of points in $\scrP_T$ is asymptotically
\begin{equation}
N(T)  \sim \pi T^2 , \qquad T\to\infty.
\end{equation}
We are interested in the distribution of directions $\|\vecy\|^{-1} \vecy$ as $\vecy$ ranges over $\scrP_T$, counted {\em with} multiplicity. That is, if there are $k$ lattice points corresponding to the same direction, we will record that direction $k$ times. For each $T$, this produces a finite sequence of $N(T)$ unit vectors $(\cos(2\pi\xi_n),\sin(2\pi\xi_n))$ with $\xi_n=\xi_n(T)\in\T=\R/\Z$ and $n=1,\ldots,N(T)$. Here we interpret the definition of sequence rather loosely: we are content with a dense set of angles that is uniformly distributed when exhausted by the radius $T$. For any interval $U\subset\T$, we have
\begin{equation}\label{udi}
\lim_{T\to\infty} \frac{\# \{ n \le N(T) :  \xi_n \in U \}}
{N(T)} 
= |U| ,
\end{equation}
where $|\cdot|$ denotes length. 
Defining $X_N(x,I)$ as in \eqref{eq:defXf}, eq.~\eqref{udi} implies that for any Borel probability measure $\lambda$ on $\T$ with continuous density,
\begin{equation}\label{expectation}
\lim_{t\to\infty} \int_\T X_{N(T)} (x,I) \, \lambda(dx) = | I|. 
\end{equation}

It is proved in \cite{marklof_strombergsson_free_path_length_2010} that for every $\vecxi\in\R^2$ and $x \in\T$ random with respect to $\lambda$ (which is only assumed to be absolutely continuous with respect to the Lebesgue measure), the random variable $X_{N(T)} (\cdot,I)$ has a limit distribution $X_\infty(\cdot,I)$. That is, for every $r\in\Z_{\ge 0}$,
\begin{equation}\label{eq:one_interval_directions}
\lim_{T\to\infty}\lambda(\{x \in\T : X_{N(T)} (x,I) = r\}) = X_\infty(r,I).
\end{equation}
The limit distribution $X_\infty(\cdot,I)$ is independent of the choice of $\lambda$, $\scrL$, and, if $\vecxi\notin\Q^2$, independent of $\vecxi$. In fact, these results hold for several test intervals $I_1,\ldots, I_m$, and follow directly from Theorem 6.3, Remark 6.4  and Lemma 9.5 of \cite{marklof_strombergsson_free_path_length_2010} 
for $\vecxi\notin\Q^2$ and from Theorem 6.5, Remark 6.6 and Lemma 9.5 of \cite{marklof_strombergsson_free_path_length_2010}  in the case $\vecxi\in\Q^2$:

\begin{theorem}[Marklof, Str\"ombergsson \cite{marklof_strombergsson_free_path_length_2010}]
\label{th:prelim_directions}
Fix $\vecxi\in\R^2$ and let $I=I_1\times\cdots \times I_m\subset \R^m$ be a bounded box. Then there is a probability distribution $X_\infty(\cdot,I)$ on $\Z_{\ge 0}^m$ such that, for any $\vecr=(r_1,\ldots,r_m)\in\Z_{\ge 0}^m$ and any Borel probability measure $\lambda$ on $\T$, absolutely continuous with respect to Lebesgue,
\begin{equation}
\lim_{T\to\infty}\lambda(\{ x\in\T : X_N(x,I_1)= r_1,\ldots,X_N(x,I_m)=r_m\}) = X_\infty(\vecr,I).
\end{equation}
\end{theorem}

In the case of rational  $\vecxi$,  an error term is easily obtained since the proof uses mixing on a finite cover of $\sltz\quot \sltr$. 
Owing to recent work of Str\"ombergsson \cite{strombergsson_effective_2013}, convergence can also be made effective for $\vecxi\notin\Q^2$, with rate depending on Diophantine properties of $\vecxi$. 

In the language of point processes, Theorem \ref{th:prelim_directions} says that the point process \[\{ N(T) (\xi_n-x)\}_{n\le N(T)}\] on the torus $\R/(N(T)\Z)$ converges, as $T\to\infty$, to a random point process on $\R$ which is determined by the probabilities $X_\infty(\vecr,I)$, thus answering the question of convergence of local statistics for this sequence. 
We highlight some key properties proven in \cite{distribution_of_directions}:
\begin{enumerate}[(a)]
\item $X_\infty(\vecr,I)$ is independent of $\lambda$ and $\scrL$.
\item $X_\infty(\vecr,I+t\underline{e})=X_\infty(\vecr,I)$ for any $t\in\R$, where $\underline{e}=(1,1,\ldots,1)$; that is, the limiting process is translation invariant.
\item $\sum_{\vecr\in\Z_{\ge 0}^m} r_j X_\infty(\vecr,I) = \sum_{k=0}^\infty r X_\infty(r,I_j) = |I_j|$ for any $j\le m$.
\item For $\vecxi\in\Q^2$, $\sum_{\vecr\in\Z_{\ge 0}^m} \|\vecr\|^ s  X_\infty(\vecr,I) < \infty$ for $0\le  s <2$, and $=\infty$ for $ s \ge 2$. 
\item For $\vecxi\notin\Q^2$, $X_\infty(\vecr,I)$ is independent of $\vecxi$.
\item For $\vecxi\notin\Q^2$, $\sum_{\vecr\in\Z_{\ge 0}^m} \|\vecr\|^ s  X_\infty(\vecr,I) < \infty$ for $0\le  s <3$, and $=\infty$ for $ s \ge 3$. 
\end{enumerate}

Properties (d) and (f) imply that the limiting process is \emph{not} a Poisson process. We will however see that  when $\vecxi\notin\Q^2$, the second moments and two-point correlation functions are those of a Poisson process with intensity $1$. Specifically, we have
\begin{equation}\label{eq:poisson2directions}
\sum_{\vecr\in\Z_{\ge 0}^2} r_1 r_2 X_\infty(\vecr,I_1\times I_2) = |I_1\cap I_2|+|I_1|\,|I_2|
\end{equation}
and, in particular, 
\begin{equation}\label{eq:poisson1directions}
\sum_{r=0}^\infty r^2 X_\infty (r,I_1) = |I_1|+|I_1|^2 ,
\end{equation}
which coincide with the corresponding formulas for the Poisson distribution. 

The problem we discuss in this section is to establish the convergence of moments to the finite moments of the limiting process. It is interesting that the convergence of certain moments requires a Diophantine condition on $\vecxi$. We say that  \emph{$\vecxi \in \R^2$ is Diophantine of type $\kappa$} if there exists $C>0$  such that 
\beq\forall \bm k = (k_1, k_2) \in \Z^2 \setminus \{\vecnull\}, \forall \ell \in \Z, \, |\bm k\cdot \vecxi + \ell| \ge \frac C{(|k_1|+|k_2|)^{\kappa}}.\eeq
It is well known that Lebesgue almost all $\vecxi\in\R^2$ are Diophantine of type $\kappa>2$, and that there is no $\vecxi\in\R^2$ which is Diophantine of type $\kappa<2$ \cite{schmidt_diophantone_1980}.
A specific example of a Diophantine vector of type $\kappa=2$ can be obtained from a degree 3 extension $K$ over $\Q$: If $\xi^1, \xi^2\in K$ are such that $\{1,\xi^1,\xi^2\}$ is a $\Q$-basis for $K$, then $\vecxi=(\xi^1,\xi^2)$ is Diophantine of type $2$ (see Theorem III of Chapter 5 and its proof in \cite{Cassels57DiophantineApproximation}). 

We also recall that $\omega\in\R$ is \emph{Diophantine of type $\kappa$} if 
\beq\forall k \in \Z \setminus \{0\}, \forall \ell \in \Z, \, |k \omega + \ell| \ge \frac C{|k|^{\kappa}}.\eeq
Here the critical value of $\kappa $ is $1$: almost all real numbers are Diophantine of type $\kappa>1$ and none are Diophantine of type $\kappa<1$. Numbers with bounded entries in the continued fraction expansion and, in particular, quadratic irrationals like $\sqrt 2$ achieve $\kappa=1$. 

For $I=I_1\times\cdots\times I_m\subset  \R^m$, $\lambda$ a Borel probability measure on $\T$ and $\vecs=(s_1,\ldots,s_m)\in\R_{\ge0}^m$ let
\begin{equation}
\M_\lambda(T,\vecs):=\int_\T \left(X_{N(T)}(x,I_1)\right)^{s_1} \cdots \left(X_{N(T)}(x,I_m)\right)^{s_m} \lambda(dx).
\end{equation}

\begin{theorem}[El-Baz, Marklof, V. {\cite[Th.~2]{distribution_of_directions}}]
\label{th:main_directions}
Let $I=I_1\times\cdots\times I_m\subset  \R^m$ be a bounded box, and $\lambda$ a Borel probability measure on $\T$ with continuous density. Choose $\vecxi\in\R^2$ and $\vecs=(s_1,\ldots,s_m)\in\R_{\ge0}^m$, such that one of the following hypotheses holds:
\begin{enumerate}[{\rm({A}1)}]
\item  $s_1+\ldots+s_m< 2$.
\item  $\vecxi$ is Diophantine of type $\kappa$, and $s_1+\ldots+s_m<2+\frac2{\kappa}$. 
\item  $\vecxi=\vecn \omega +\vecl$ where $\vecn\in\Z^2 \setminus \{ \vecnull \}$ and $\vecl\in\Q^2$ so that $\det(\vecn,\vecl)\notin\Z$, and $\omega\in\R$ is Diophantine of type $\frac{\kappa}{2}$, and $s_1+\ldots+s_m<2+\frac2{\kappa}$.
\end{enumerate}
Then,
\begin{equation}\label{eq:main_directions}
\lim_{T\to\infty} \M_\lambda(T,\vecs) = \sum_{\vecr\in\Z_{\ge 0}^m} r_1^{s_1}\cdots r_m^{s_m} X_\infty(\vecr,I).
\end{equation}
\end{theorem}

The fact that \emph{some} Diophantine condition is necessary in (A2) or (A3) can be seen from the following argument. Assume that $\veck\cdot(\vecxi+\vecm) = 0$ for some $\veck\in\Z^2\setminus\{\vecnull\}$, $\vecm\in\Z^2$. Then there is a line through the origin (in direction $\alpha_\veck$, say) that contains infinitely many lattice points of $\scrL_\vecxi$ so that, for any $\epsilon>0$ and sufficiently large $T$,
\begin{equation}
X_{N(T)}(\alpha_\veck,(-\epsilon,\epsilon)) \gg_{\veck,\scrL} T,
\end{equation}
where the implied constant depends only on $\veck$ and $\scrL$. This in turn implies that when $\lambda $ is the Lebesgue measure and $s\ge 2$ we have
\begin{equation}
\M_{\leb}(T,s) \gg_{\veck,\scrL}  T^{s-2}, 
\end{equation}
and thus any moment with $s>2$ diverges. In the case $s=2$, we have for any bounded interval $I\subset\R$
\begin{equation}\label{limit:main111}
\liminf_{T\to\infty} \M_{\leb}(T,2) > \sum_{r\in\Z_{\ge 0}^m} r^{2} X_\infty(r,I).
\end{equation}

The condition (A3) in Theorem \ref{th:main_directions} comes from the realization that a natural obstruction to convergence of moments is collinearity of $\vecxi$ with points of $\Z^2$. This situation is indeed ruled out by the condition: while it is certain that $\vecxi$ lies on a rational line, this rational line intersects $(\Z/d)^2$ for some $d\ge 2$, but not $\Z^2$. For example, $\vecxi=(\sqrt 2 +1/2,\sqrt 2 + 1)$ lies on the rational line $2x-2y=-1$, but this line clearly misses all the lattice points. 

The proof of Theorem \ref{th:main_directions} builds on the proof of Theorem \ref{th:prelim_directions}. In the proof, random variables $X_{N(T)}(\cdot,I)$ are approximately realized as a fixed function on a certain homogeneous space equipped with a $T$-dependent probability measure. The result (Theorem \ref{th:prelim_directions}) then follows from weak convergence of these probability measures, which means that integrals of a bounded continuous function with respect to these measures tend to the integral with respect to the limit measure. In fact to prove Theorem \ref{th:prelim_directions}, one needs to use functions which are bounded but not quite continuous. This is not a problem since the set of discontinuities of the these functions is small. To prove Theorem \ref{th:main_directions}, however, we need to use functions that are unbounded, which is a substantial complication. 

To explain the key step in the proof of Theorem \ref{th:main_directions}, define the restricted moments
\begin{equation}
\M_\lambda^{(K)}(T,\vecs):=\int_{\max_j X_{N(T)}(x,I_j)\le K}\limits\left(X_{N(T)}(x,I_1)\right)^{s_1} \cdots \left(X_{N(T)}(x,I_m)\right)^{s_m} \lambda(dx) .
\end{equation}
Theorem \ref{th:prelim_directions} now implies that, for any $K\ge 0$,
\begin{equation}\label{eq:main_directions2}
\lim_{T\to\infty} \M_\lambda^{(K)}(T,\vecs) = \sum_{\substack{\vecr\in\Z_{\ge 0}^m\\ |\vecr|\le K}} r_1^{s_1}\cdots r_m^{s_m} X_\infty(\vecr,I),
\end{equation}
where $|\vecr|$ denotes the maximum norm of $\vecr$. What thus remains to be shown in the proof of Theorem \ref{th:main_directions} is that under (A1), (A2), and (A3),
\begin{equation}\label{eq:main_directions3}
\adjustlimits\lim_{K\to\infty} \limsup_{T\to\infty} \left|\M_\lambda(T,\vecs)-\M_\lambda^{(K)}(T,\vecs)\right| = 0 .
\end{equation}

With \eqref{eq:poisson2directions}, Theorem \ref{th:main_directions} has the following implications:

\begin{cor}\label{Cor1}
Let $I=I_1\times I_2\subset  \R^2$ and $\lambda$ be as in Theorem \ref{th:main_directions}, and assume $\vecxi\in\R^2$ is Diophantine. Then
\begin{equation}
\lim_{T\to\infty} \int_\T X_{N(T)}(x,I_1) \, X_{N(T)}(x,I_2)\, \lambda(dx) = |I_1\cap I_2|+|I_1|\;|I_2| .
\end{equation}
\end{cor}
%

With pair correlation defined as in \eqref{eq:pair_def}, we have

\begin{cor}\label{cor2}
Assume $\vecxi\in\R^2$ is Diophantine. Then, for any $f\in C_0(\R)$
\begin{equation}\label{eq:paircon_directions}
\lim_{T\to\infty}  R_{N(T)}^2(f) = \int_{\R} f(s)\, ds.
\end{equation}
\end{cor}

This answers a recent question by Boca, Popa and Zaharescu \cite{boca_pair_2013}. Figure \ref{fig:pair} shows a numerical computation of the pair correlation statistics for $\vecxi=(\sqrt[3]4,\sqrt[3]2)$, $T=1000$, which is close to the limiting density $1$ predicted by Corollary \ref{cor2}. 

\begin{figure}[t]
\centerline{\includegraphics[width=0.9\textwidth]{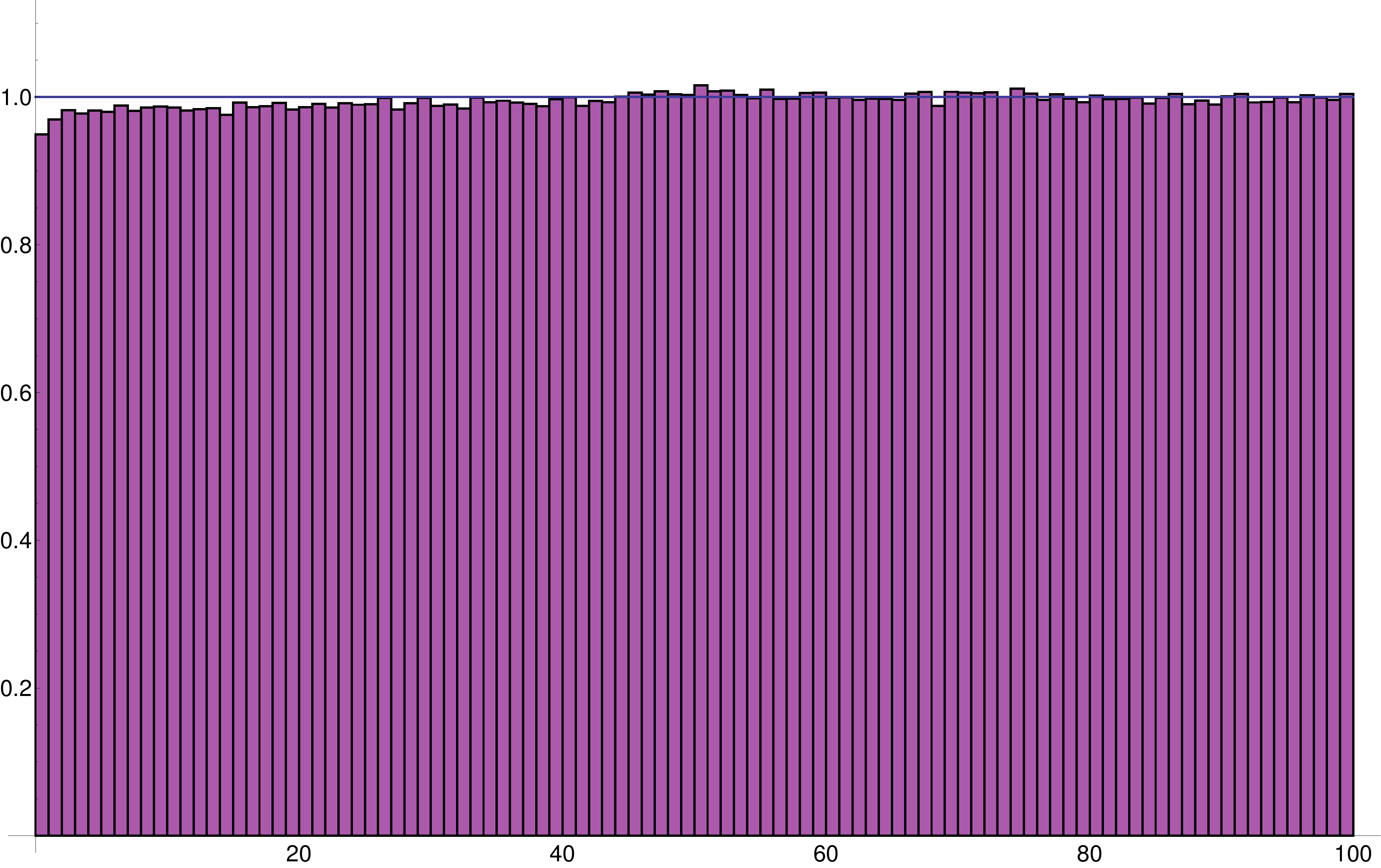}}
\caption{The figure shows a numerical computation of the pair correlation density, for $\vecxi=(\sqrt[3]4,\sqrt[3]2)$, $T=1000$. The computed density is close to $1$, as predicted by Corollary \ref{cor2}. Note that the displayed histogram can be obtained as the sum over all $k$-neighbor distributions. }
\label{fig:pair}
\end{figure}

%

In the next sections we explain how to construct a fixed function and a sequence of probability measures on a homogeneous space to realize $X_N(\cdot,I)$, as well as outline some ideas of the proof that allows the use of slowly growing functions in an equidistribution theorem.

\subsection{Space of affine lattices}

Let $G=\SL(2,\R)$ and $\Gamma=\SL(2,\Z)$. Define $G'=G\ltimes \R^2$ by 
\beq (M,\vecxi)(M',\vecxi')=(MM',\vecxi M'+\vecxi'),\eeq
and let $\Gamma'=\Gamma\ltimes \Z^2$ denote the integer points of this group. 
In the following, we will embed $G$ in $G'$ via the homomorphism $M\mapsto (M,\vecnull)$ and identify $G$ with the corresponding subgroup in $G'$.
We will refer to the homogeneous space $\Gamma\quot G$ as the space of lattices and $\Gamma'\quot G'$ as the space of affine lattices. The natural right action of $G'$ on $\R^2$ is given by  $\vecx\mapsto\vecx (M,\vecxi) := \vecx M + \vecxi$, with $(M,\vecxi)\in G'$.


Given a bounded interval $I\subset\R$, define the triangle
\begin{equation}\label{CcI}
\triangle(I) =\{ (x,y)\in\R^2 : 0< x<1, \, y\in 2 x I   \}
\end{equation}
and set, for $g\in G'$ and any bounded subset $S\subset\R^2$,
\begin{equation}
X(g,S)= \# ( S \cap \Z^2 g) .
\end{equation}
By construction, $X(\cdot,S)$ is a function on the space of affine lattices, $\Gamma'\quot G'$.

Let 
\begin{equation}
\Phi^t =\begin{pmatrix} {e}^{-t/2} & 0 \\ 0 & {e}^{t/2} \end{pmatrix}, \qquad
k(\phi) = \begin{pmatrix} \cos\phi & -\sin\phi \\ \sin\phi & \cos\phi \end{pmatrix}.
\end{equation}

\begin{figure}
\centerline{\includegraphics[width=0.7\textwidth]{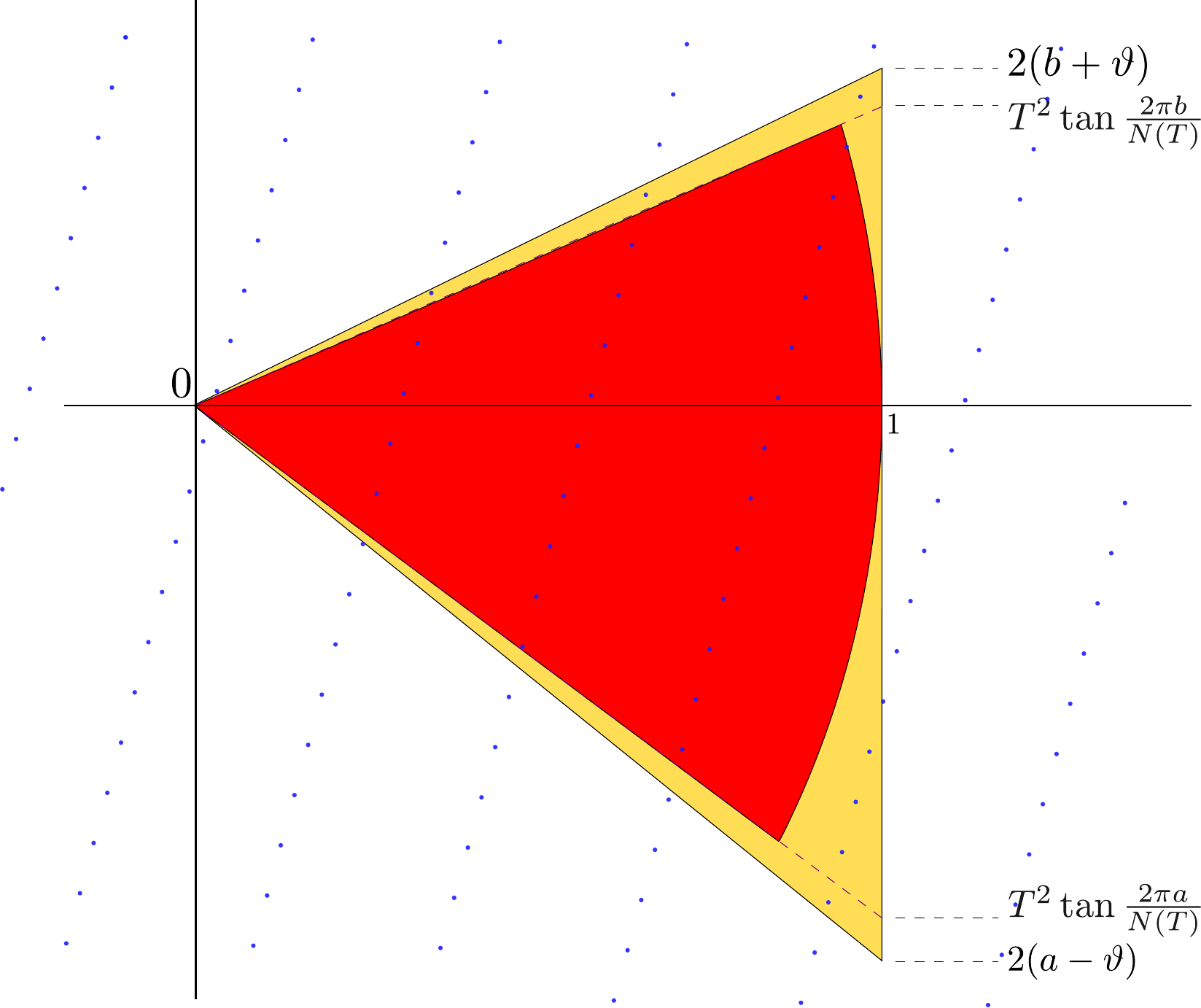}}
\caption{Here $I=[a,b]$ with $a<0<b$. The dark (red) area corresponds to counting in $X_{N(T)} (x,I)$, while the grey (yellow) triangle is the bound we use in \eqref{eq:crude_directions}.}
\label{fig:triangle}
\end{figure}

An elementary geometric argument shows that, given $I\subset\R$ and $\vartheta>0$, there exists $T_0>0$ such that for all $x \in\T$, $\vecxi\in\R^2$, $M_0\in \Gamma'\quot G'$ and $T={e}^{t/2}\ge T_0$,
\begin{equation}\label{eq:crude_directions}
X_{N(T)} (x,I) \le X \left((1,\vecxi) M_0 k(2\pi x)\Phi^{t}, \triangle(I + [-\vartheta,\vartheta]) \right).
\end{equation}
Indeed, the quantity on the left hand side counts the number of lattice points in a cone, while that on the right hand side counts lattice points in a triangle that properly contains the closure of this set. This is illustrated in Figure~\ref{fig:triangle}.
The observation \eqref{eq:crude_directions} relates our original counting function $X_{N(T)} (x,I)$ to a function on the space of lattices. Since we will only require upper bounds, the crude estimate \eqref{eq:crude_directions} is sufficient. A more refined statement is used in \cite[Sect.~9.4]{marklof_strombergsson_free_path_length_2010}, where sets $\triangle^{(t)}(I)$ are constructed such that $X_{N(T)} (x,I) = X\left((1,\vecxi) k(2\pi x)\Phi^{t}, \triangle^{(t)}(I) \right)$ and the sequence of sets $\triangle^{(t)}(I)$ converges to $\triangle(I)$  as $t\to\infty$. 

A convenient parametrization of $M\in G$ is given by the the Iwasawa decomposition
\begin{equation}\label{iwasawa}
M=n(u)a(v)k(\varphi)
\end{equation}
where
\begin{equation}
n(u)=\begin{pmatrix} 1 & u \\ 0 & 1 \end{pmatrix},\qquad
a(v)=\begin{pmatrix} v^{1/2} & 0 \\ 0 & v^{-1/2} \end{pmatrix} ,
\end{equation}
with $\tau=u+{i} v$ in the complex upper half plane $\H=\{ u+{i} v \in\C : v>0\}$ and $\phi\in[0,2\pi)$. 
A convenient parametrization of $g\in G'$ is then given by $\H\times [0,2\pi) \times \R^2$ via the decomposition
\begin{equation}\label{ida}
g = (1,\vecxi) n(u)a(v)k(\varphi)\eqqcolon(\tau,\phi;\vecxi).
\end{equation}
In these coordinates, left multiplication on $G$ becomes the (left) group action
\begin{equation}
g\cdot (\tau,\varphi;\vecxi)=(g\tau,\varphi_g;\vecxi g^{-1})
\end{equation}
where for 
\begin{equation}
g =(1,\vecm)\abcd
\end{equation}
we have:
\begin{equation}
g\tau = u_g+{i} v_g = \frac{a\tau+b}{c\tau+d}
\end{equation}
and thus
\begin{equation}
v_g=\im (g \tau)=\frac{v}{|c\tau+d|^2};
\end{equation}
furthermore
\begin{equation}
\varphi_g=\varphi+\arg(c\tau+d),
\end{equation}
and
\begin{equation}
\vecxi g^{-1} = (d\xi^1-c\xi^2,-b\xi^1+a\xi^2) - \vecm  .
\end{equation}

The space of lattices has one cusp, which in the above coordinates appears at $v\to\infty$. The following lemma tells us that $X(g,S)$ is bounded in the cusp unless $-\xi^1$ is close to an integer, in which case the function is at most of order $v^{1/2}$.

\begin{lemma}\label{lem:upper}
For any bounded $S\subset\R^2$, $g=(1,\vecxi)(M,0)\in G'$ with $M$ as in \eqref{iwasawa} and $v\ge 1$,
\begin{equation}\label{eq101}
X(g,S) \le (2 r v^{1/2}+1) \, \# ((\Z+\xi^1) \cap [-rv^{-1/2},rv^{-1/2}])
\end{equation}
where $r=\sup \{\|\vecx\| : \vecx\in S \}$. If $v> 4 r^2$ then, for any $ s \ge 0$,
\begin{equation}\label{eq102}
\left(X(g,S)\right)^ s  \le (2 r v^{1/2}+1)^ s  \, \# ((\Z+\xi^1) \cap [-rv^{-1/2},rv^{-1/2}]).
\end{equation} 
\end{lemma}

\begin{proof}
Let $D_r$ be the smallest closed disk of radius $r$ centered at zero which contains $S$. Then,
\begin{equation}
\begin{split}
X(g,S) & \le X(g,D_r) \\
& = \# ( D_r \cap (\Z^2+\vecxi) n(u) a(v) ) \\
& \le \# ( [-r,r]^2 \cap (\Z^2+\vecxi) n(u) a(v) ) \\
& = \# ( ([-rv^{-1/2},rv^{-1/2}]\times[-rv^{1/2},rv^{1/2}]) \cap (\Z^2+\vecxi) n(u) ) \\
& \le \sup_{\xi^2} \# ( ([-rv^{1/2},rv^{1/2}]) \cap (\Z+\xi^2) )\times \# ( [-rv^{-1/2},rv^{-1/2}] \cap (\Z+\xi^1) ) \\
& \le (2r v^{1/2}+1)\times \# ( [-rv^{-1/2},rv^{-1/2}] \cap (\Z+\xi^1) ) .
\end{split}
\end{equation}
This proves \eqref{eq101}. The second inequality \eqref{eq102} follows from the fact that $\# ((\Z+\xi^1) \cap [-rv^{-1/2},rv^{-1/2}])\in\{0,1\}$.
\end{proof}

To deal with the case of mixed moments, we note that 
\begin{equation}
\left(X(g,S_1)\right)^{ s _1}\cdots \left(X(g,S_m)\right)^{ s _m}\le \left(X(g,S_1\cup\cdots\cup S_m)\right)^{ s _1+\ldots+ s _m}. 
\end{equation}
%
%
%
%

\subsection{Escape of mass}

We define the abelian subgroups
\[\Gamma_\infty= \left\{\begin{pmatrix}1 &m\\ 0 &1\end{pmatrix}: m\in\Z\right\}\subset\Gamma  \]
and
\[\Gamma_\infty'= \left\{\left(\begin{pmatrix}1 &m_1\\ 0 &1\end{pmatrix}, (0,m_2) \right): (m_1,m_2)\in\Z^2\right\}\subset\Gamma' . \]  
These subgroups are the stabilizers of the cusp at $\infty$ of $\Gamma\quot G$ and  $\Gamma'\quot G'$, respectively.

Denote by $\chi_R$ the characteristic function of $[R,\infty)$ for some $R\ge 1$, i.e.~$\chi_R(v)=0$ if $v<R$ and $\chi_R(v)=1$ if $v\ge R$. 
For a fixed real number $\beta$ and a continuous function $f:\R\to\R$ of rapid decay at $\pm\infty$, define the function $F_{R,\beta}\colon \H\times\R^2\to \R$ by 
\begin{equation}\label{eq:bounding_function}
\begin{split}
F_{R,\beta}\left(\tau; \vecxi\right)&=\sum_{\gamma\in\Gamma_\infty\quot \Gamma}\sum_{m\in\Z}f (((\vecxi\gamma^{-1})_1+m)v^{1/2}_\gamma) v^\beta_\gamma\chi_R(v_\gamma) \\
&=\sum_{\gamma\in\Gamma_\infty'\quot \Gamma'} f_\beta(\gamma g) ,
\end{split}
\end{equation}
where $f_\beta:G'\to\R$ is defined by 
\begin{equation}
f_\beta((1,\vecxi) n(u)a(v)k(\varphi)) := f(\xi^1 v^{1/2}) v^\beta \chi_R(v).
\end{equation}
We view $F_{R,\beta}\left(\tau; \vecxi\right)=F_{R,\beta}\left(g\right)$ as a function on $\Gamma'\quot G'$ via the identification \eqref{ida}. 

The main idea behind the definition of $F_{R,\beta}\left(\tau; \vecxi\right)$ is that we have for $v\ge 1$
\begin{equation}
F_{R,\beta}\left(\tau; \vecxi\right) = \sum_{m\in\Z} [f((\xi^1+m)v^{1/2})+f((-\xi^1+m)v^{1/2})] v^\beta \chi_R(v) ,
\end{equation}
which shows that, for an appropriate choice of $f$ depending on $S_1,\dots,S_m$ and $\beta=\frac12( s _1+\ldots+ s _m)$, and $v\ge R$ with $R$ sufficiently large,
\begin{equation}\label{eq:Fcontrol_directions}
\left(X(g,S_1)\right)^{ s _1}\cdots \left(X(g,S_m)\right)^{ s _m} \le F_{R,\beta}\left(\tau; \vecxi\right).
\end{equation}
Therefore $F_{R,\beta}$ is the fixed function that controls moments. 

The following proposition establishes under which conditions there is no escape of mass in the equidistribution of horocycles. It generalizes results in \cite{marklof_pair_correlation_II_2002, marklof_pair_correlation_2003, marklof_mean_square_2005}.

\begin{prop}[El-Baz, Marklof, V. {\cite[Prop.~6]{distribution_of_directions}}]
\label{prop:main_homogeneous_directions}
Let $\vecxi\in\R^2$, $\beta\ge 0$, $M\in G$, and $h\in C_0(\R)$. Assume that one of the following hypotheses holds:
\begin{enumerate}[{\rm ({B}1)}]
\item  $\beta<1$.
\item $\vecxi$ is Diophantine of type $\kappa$, and $\beta<1+\frac1{\kappa}$. 
\item $\vecxi=\vecn \omega +\vecl$ where $\vecn\in\Z^2 \setminus \{ \vecnull \}$ and $\vecl\in\Q^2$ so that $\det(\vecn,\vecl)\notin\Z$, and $\omega\in\R$ is Diophantine of type $\frac{\kappa}{2}$, and $\beta<1+\frac1{\kappa}$.
\end{enumerate}
Then
\beq \label{eq:horospherical}
\adjustlimits\lim_{R\to \infty}\limsup_{v\to 0} \bigg| \int_{u\in\R} F_{R,\beta}\left((1,\vecxi)Mn(u)a(v)\right) h(u)du \bigg|=0.
\eeq
\end{prop}


The proof is broken up into several parts. When $\beta<1$, we control $F_{R,\beta}$ by a function that is independent of $\vecxi$ and use Eisenstein series to control the integral. Under assumptions (B2) and (B3), we first consider the case $M=1$; this is the bulk of the proof. Here we write out the definition of $F_{R,\beta}$ at the relevant point in $\Gamma'\quot G'$ and prove a Lemma that uses the Diophantine condition on $\vecxi$ and eventually lets us control excursions to the cusp. More calculation lets us take general $M$ in the statement as well as replace the horospherical average in \eqref{eq:horospherical} by a spherical one, which is what controls points in a large Euclidean ball. 

\section{$\sqrt n$ modulo 1\label{sec:sqrtn}}


\subsection{Setup}

\begin{figure}
\begin{centering}\includegraphics[width=0.9\textwidth]{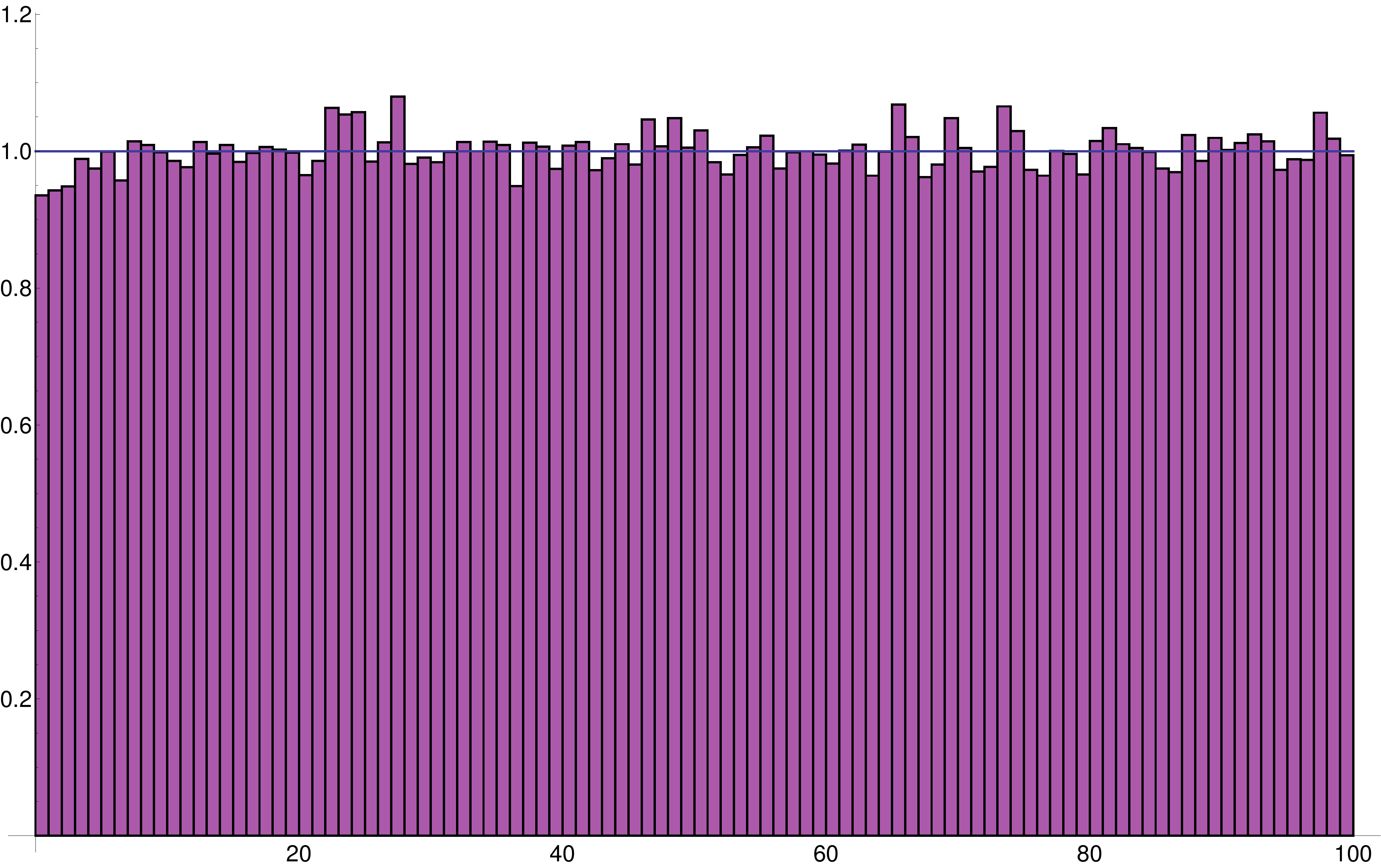}\par\end{centering}
\caption{Two-point correlations of the fractional parts of $\sqrt n$ with $n\leq 2000$, $n\notin\Box$.  \label{fig3}}
\end{figure}

In this section we analyze local statistics of $\xi_n=\sqrt n\bmod 1$ following the treatment in \cite{el-baz_two-point_2013}.
To describe our results, let us first note that $\sqrt{n}=0\bmod 1$ if and only if $n$ is a perfect square. We will remove this trivial subsequence and consider the set
 \begin{equation}
 \scrP_T=\{\sqrt n \bmod 1 : 1\le n\le T,\; n\notin\Box\}\subset \T:=\R/\Z
 \end{equation}
 where $\Box\subset\N$ denotes the set of perfect squares. The cardinality of $\scrP_T$ is $N(T)=T-\lfloor \sqrt T\rfloor$. We label the elements of $\scrP_T$ by $\xi_1,\ldots,\xi_{N(T)}$. The pair correlation density $R_{N(T)}^2(f)$ of the $\xi_j$ is defined by as in \eqref{eq:pair_def}
where $f\in C_0(\R)$ (continuous with compact support). 
Our first result establishes that $R_{N(T)}^2$ converges weakly to the two-point density of a Poisson process:

\begin{theorem}\label{th:pair_sqrt}
For any $f\in C_0(\R)$,
\begin{equation}\label{eq:paircon_sqrt}
\lim_{T\to\infty}  R_{N(T)}^2(f) = \int_{\R} f(s)\, ds.
\end{equation}
\end{theorem}
%

It is proved in \cite{ElkiesMcM04} that, for $x$ uniformly distributed in $\T$ with respect to the Lebesgue measure $\lambda$, the random variable $X_{N(T)}(x,I)$ has a limit distribution $X_\infty(\cdot,I)$. That is to say, for every $r\in\Z_{\geq 0}$,
\begin{equation}\label{eq:one_interval_sqrt}
\lim_{T\to\infty}\lambda(\{x\in\T : X_{N(T)}(x,I) = r\}) = X_\infty(r,I).
\end{equation}
As Elkies and McMullen point out, these results hold in fact for several test intervals $I_1,\ldots, I_m$:

\begin{theorem}[Elkies and McMullen \cite{ElkiesMcM04}] \label{th:prelim}
Let $I=I_1\times\cdots \times I_m\subset \R^m$ be a bounded box. Then there is a probability distribution $X_\infty(\cdot,I)$ on $\Z_{\geq 0}^m$ such that, for any $\vecr=(r_1,\ldots,r_m)\in\Z_{\geq 0}^m$
\begin{equation}
\lim_{T\to\infty}\lambda(\{ x\in\T : X_{N(T)}(x,I_1) = r_1,\ldots,X_{N(T)}(x,I_m) =r_m\}) = X_\infty(\vecr,I).
\end{equation}
\end{theorem}

Theorem \ref{th:prelim} states that the point process \[\{ N(T) (\xi_j-x)\}_{j\le N(T)}\] on the torus $\R/(N(T)\Z)$ converges, as $T\to\infty$, to a random point process on $\R$ which is determined by the probabilities $X_\infty(\vecr,I)$. As pointed out in \cite{marklof_strombergsson_free_path_length_2010}, this process is the same as for the directions of affine lattice points with irrational shift (see Section \ref{sec:directions}). It is described in terms of a random variable in the space of affine lattices and is in particular not a Poisson process. The second moments and two-point correlation function, however, coincide with those of a Poisson process with intensity $1$. 

It is important to note that Elkies and McMullen considered the full sequence $\{ \sqrt{n}\bmod 1 : 1\leq n \leq T\}$. Removing the perfect squares $n\in\Box$ does not have any effect on the limit distribution in Theorem  \ref{th:prelim}, since the set of $x$ for which  $X_{N(T)}(x,I)$ is different has vanishing Lebesgue measure as $T\to\infty$. In the case of the second and higher moments, however, the removal of perfect squares will make a difference and in particular avoid trivial divergence. 

For $I=I_1\times\cdots\times I_m\subset  \R^m$ and $\vecs=(s_1,\ldots,s_m)\in\R_{\ge 0}^m$ let
\begin{equation}
\M(T,\vecs):=\int_\T (X_{N(T)}(x,I_1))^{s_1} \cdots (X_{N(T)}(x,I_m))^{s_m} \, dx.
\end{equation}
The main objective of this section is to explain the convergence of these mixed moments to the corresponding moments of the limit process, where they exist. The case of the second mixed moment implies, by a standard argument, the convergence of the two-point correlation function stated in Theorem \ref{th:pair_sqrt}, cf.\ Appendix 1 of \cite{distribution_of_directions} and Lemma \ref{lem:X_to_pair}.

\begin{theorem}[El-Baz, Marklof, V. {\cite[Th.~3]{el-baz_two-point_2013}}]
\label{th:main0}
Let $I=I_1\times\cdots\times I_m\subset  \R^m$ be a bounded box, and $\lambda$ a Borel probability measure on $\T$ with continuous density. Choose  $\vecs=(s_1,\ldots,s_m)\in\R_{\ge0}^m$, such that $s_1+\dots+s_m<3$. Then,
\begin{equation}\label{eq:main_sqrt}
\lim_{T\to\infty} \M(T,\vecs) = \sum_{\vecr\in\Z_{\geq 0}^m} r_1^{s_1}\cdots r_m^{s_m} X_\infty(\vecr,I).
\end{equation}
\end{theorem}

\subsection{Strategy of proof}

The proof of Theorem \ref{th:main0} follows our strategy in the case of lattice translates \eqref{eq:main_directions2} and \eqref{eq:main_directions3}. We define the restricted moments
\begin{equation}
\M^{(K)}(T,\vecs):=\int_{\max_j X_{N(T)}(x,I_j)\leq K} \limits (X_{N(T)}(x, I_1))^{s_1} \cdots (X_{N(T)}(x,I_m))^{s_m}  dx .
\end{equation}
Theorem \ref{th:prelim} implies that, for any fixed $K\geq 0$,
\begin{equation}\label{eq:main_sqrt2}
\lim_{T\to\infty} \M^{(K)}(T,\vecs) = \sum_{\substack{\vecr\in\Z_{\geq 0}^m\\ |\vecr|\leq K}} r_1^{s_1}\cdots r_m^{s_m} X_\infty(\vecr,I),
\end{equation}
where $|\vecr|$ denotes the maximum norm of $\vecr$. To prove Theorem \ref{th:main0}, what remains is to show that 
\begin{equation}\label{eq:main_sqrt3}
\adjustlimits\lim_{K\to\infty} \limsup_{T\to\infty} \left|\M(T,\vecs)-\M^{(K)}(T,\vecs)\right| = 0 .
\end{equation}
To establish the latter, we use the inequality
\begin{equation}
\left|\M(T,\vecs)-\M^{(K)}(T,\vecs)\right|
\leq \int_{X_{N(T)}(x, \overline I)\geq K} (X_{N(T)}(x, \overline I))^{ s } dx
\end{equation}
where $\overline I=\cup_j I_j$ and $ s =\sum_j s_j$. As in the work of Elkies and McMullen, the integral on the right hand side can be interpreted as an integral over a translate of a non-linear horocycle in the space of affine lattices. The main difference is that now the test function is unbounded, and we require an estimate that guarantees there is no escape of mass as long as $ s <3$. This means that 
\begin{equation}\label{upperb}
\adjustlimits\lim_{K\to\infty} \limsup_{T\to\infty} \int_{X_{N(T)}(x, \overline I) \geq K} (X_{N(T)}(x, \overline I))^{ s } dx =0
\end{equation}
implies Theorem \ref{th:main0}.

\subsection{Escape of mass in the space of lattices\label{sec:affine}}
We proceed as in Section \ref{sec:directions}. 
Let $G=\SL(2,\R)$ and $\Gamma=\SL(2,\Z)$. Define the semi-direct product $G'=G\ltimes \R^2$ by 
\beq (M,\vecxi)(M',\vecxi')=(MM',\vecxi M'+\vecxi'),\eeq
and let $\Gamma'=\Gamma\ltimes \Z^2$ denote the integer points of this group. 
In the following, we will embed $G$ in $G'$ via the homomorphism $M\mapsto (M,\vecnull)$ and identify $G$ with the corresponding subgroup in $G'$.
We will refer to the homogeneous space $\Gamma\quot G$ as the space of lattices and $\Gamma'\quot G'$ as the space of affine lattices. A natural action of $G'$ on $\R^2$ is defined by $\vecx\mapsto \vecx(M,\vecxi):=\vecx M +\vecxi$.

Given an interval $I\subset\R$, define the triangle
\begin{equation}
\triangle (I) = \{ (x,y)\in\R^2 : 0< x<2, \; y\in 2 xI   \} .
\end{equation}
and set, for $g\in G'$ and any bounded subset $S\subset\R^2$,
\begin{equation}
X(g,S)= \# ( S \cap \Z^2 g) .
\end{equation}
By construction, $X(\cdot,S)$ is a function on the space of affine lattices, $\Gamma'\quot G'$.

Let 
\begin{equation}
\Phi^t =\begin{pmatrix} {e}^{-t/2} & 0 \\ 0 & {e}^{t/2} \end{pmatrix}, \qquad
\tilde n(u) = \bigg( \begin{pmatrix} 1 & u \\ 0 & 1 \end{pmatrix} , \bigg(\frac{u}{2}, \frac{u^2}{4}\bigg)\bigg).
\end{equation}
Note that $\{ \Phi^t \}_{t\in\R}$ and $\{ \tilde n(u) \}_{u\in\R}$ are one-parameter subgroups of $G'$. Note that $\Gamma' \tilde n(u+2) = \Gamma' \tilde n(u)$ and hence $\Gamma' \{ \tilde n(u) \}_{u\in[-1,1)} \Phi^t$ is a closed orbit in $\Gamma'\quot G'$ for every $t\in\R$.

\begin{lemma}\label{lem:the}
Given an interval $I\subset\R$, there is $T_0>0$ such that for all $T={e}^{t/2}\geq T_0$, $x \in[-\frac12,\frac12]$:
\begin{equation}\label{eq:crude_sqrt}
X_{N(T)}(x, I) \leq X\left(\tilde n(2x)\Phi^{t}, \triangle(I) \right)+ X\left(\tilde n(-2x)\Phi^{t}, \triangle(I) \right)
\end{equation}
and, for $-\frac13 T^{-1/2}\leq x\leq \frac13 T^{-1/2}$,  
\begin{equation} \label{zerro}
X_{N(T)}(x,I)= 0.
\end{equation}
\end{lemma}

\begin{proof}
The bound \eqref{eq:crude_sqrt} follows from the more precise estimates in \cite{ElkiesMcM04}; cf.~also \cite[Sect.~4]{marklof_survey}. The second statement \eqref{zerro} follows from the observation that the distance of $\sqrt n$ to the nearest integer, with $n\leq T$ and not a perfect square, is at least $\frac12 (n+1)^{-1/2}\geq \frac12(T+1)^{-1/2}$.
\end{proof}

We show in Section \ref{sec:directions} that there is a choice of a continuous function $f\geq 0$ with compact support, such that for $\beta=\frac12  s $, and $v\geq R$ with $R$ sufficiently large, we have
\begin{equation}\label{eq:Fcontrol_sqrt}
(X(g,\triangle(I)))^ s  \leq F_{R,\beta}(g)=F_{R,\beta}\left(\tau; \vecxi\right).
\end{equation}
Here $F_{R,\beta}$ is the bounding function defined in \eqref{eq:bounding_function}. 

The following proposition establishes under which conditions there is no escape of mass in the equidistribution of translates of non-linear horocycles. In view of Lemma \ref{lem:the} and \eqref{eq:Fcontrol_sqrt}, it implies \eqref{upperb} and thus Theorem \ref{th:main0}. We write $v=1/T$ and note that $\frac{\beta}{2(\beta-1)}>\frac12$ so the choice $\eta=\frac12$ is always permitted.

\begin{prop}\label{prop:main_homogeneous_sqrt}
Assume $f$ is continuous and has compact support. Let $0\leq \beta <\frac32$. 
Then
\beq 
\adjustlimits\lim_{R\to \infty}\limsup_{v\to 0} \bigg| \int F_{R,\beta}\left(\tilde n(u)a(v)\right) du \bigg|=0
\eeq
where the range of integration is $[-1,1]$ for $\beta<1$, and $[-1,-\theta v^{\eta}]\cup [\theta v^{\eta},1]$ for $\beta\geq 1$ and any $\eta \in [0,\frac{\beta}{2(\beta-1)})$, $\theta\in(0,1)$.
\end{prop}
%

Note that the removal of an interval around zero from the range of integration is innocuous as we already know from Lemma \ref{lem:the} that $X_{N(T)}(x,I)$ vanishes there. The proof  in the regime $\beta<1$ is identical to the one in the case of directions in an affine lattice \cite{distribution_of_directions}. When $\beta\ge 1$, we need to control excursions to the cusp. This is done using a Lemma that has two inputs, both of number-theoretic origin; the first is that there are not too many solutions to the equation $d^2\equiv j\pmod c$ for a given $j$, and the second is cancellation in Gauss sums.

\parindent=0pt
\parsep=0pt
\parskip=0pt

\bibliographystyle{plain}
\bibliography{bibliography}

\small
\parindent=0pt
\parsep=0pt
\parskip=0pt

\textsc{Ilya Vinogradov, School of Mathematics, University of Bristol, Bristol BS8~1TW, U.K.} \texttt{ilya.vinogradov@bristol.ac.uk}

\end{document}